\documentclass{amsart}
\usepackage{amsmath,amssymb,amsthm,mathtools}
\usepackage{booktabs}
\usepackage{hyperref}

\newtheorem{theorem}{Theorem}[section]
\newtheorem{lemma}[theorem]{Lemma}
\newtheorem{proposition}[theorem]{Proposition}
\newtheorem{corollary}[theorem]{Corollary}
\numberwithin{equation}{section}

\newcommand{\F}{\mathbb F}
\newcommand{\one}{\mathbf 1}
\newcommand{\cO}{\mathcal O}
\newcommand{\cJ}{\mathcal J}
\newcommand{\cK}{\mathcal K}
\newcommand{\sfT}{\mathsf T}
\newcommand{\wideand}{ \quad \text{ and } \quad }

\DeclareMathOperator{\rank}{rank}
\DeclareMathOperator{\AG}{AG}

\title{Complement minimally non-totally unimodular matrices}

\author{Suyoung Choi}
\address{Department of Mathematics, Ajou University, Suwon 16499, Republic of Korea}
\email{schoi@ajou.ac.kr}

\author{Mathieu Vall\'ee}
\address{Universit\'e libre de Bruxelles, CP212, Boulevard du Triomphe, 1050 Brussels, Belgium}
\email{mathieu.vallee@protonmail.com}
\date{\today}

\thanks{
    The first author was supported by the National Research Foundation of Korea Grant funded by the Korean Government (RS-2025-00521982). 
    The second author was supported by the Fonds de la Recherche Scientifique-FNRS under Grant n°T003325F.
}

\dedicatory{In memory of Professor Klaus Truemper}

\begin{document}
\begin{abstract}
    We prove that, up to row and column permutations and complement operations, the only complement minimally non-totally unimodular matrices are the cycle matrices $C_3$ and $C_5$. 
    This settles a conjecture of Chervet, Grappe, and Vall\'ee. 
    As a consequence, every simplicial cone generated by the rows of a totally equimodular matrix admits a regular unimodular Hilbert triangulation.
\end{abstract}
\maketitle

\section{Introduction}
Totally unimodular matrices are a classical source of integral polyhedra in combinatorial optimization~\cite{Schrijver_1999, Conforti-Cornuejols-Zambelli2014book}.
The matroids represented by totally unimodular matrices are precisely the regular matroids~\cite{Tutte1958}, whose structure is described by Seymour's decomposition theorem~\cite{Seymour1980}.
Totally equimodular matrices generalize the class of totally unimodular matrices and are closely related to box-totally dual integral systems and box-TDI polyhedra~\cite{Chervet-Grappe-Robert2021,Lancini-Pisanu2026}.
Equimodular matrices appear in the study of box-perfect graphs~\cite{Chervet_Grappe_2024}.
Chervet, Grappe, and Vall\'ee proved a decomposition theorem for totally equimodular matrices of full row rank~\cite{Chervet_Grappe_Vallee_2026}.
Their decomposition contains two exceptional types of fundamental blocks, called \emph{thin} and \emph{thick te-interlaces}.
Through their \emph{cores}, these blocks correspond to complement totally unimodular matrices and complement minimally non-totally unimodular matrices, respectively.
Thus, classifying complement minimally non-totally unimodular matrices gives a complete description of the thick exceptional blocks.

Complement totally unimodular and complement minimally non-totally unimodular matrices were introduced by Truemper~\cite{Truemper1980}.
He gave necessary conditions for a matrix to be complement minimally non-totally unimodular, but left the classification open.
In his monograph~\cite{Truemper1992book}, he later gave a constructive characterization of complement totally unimodular matrices.
Truemper expected complement minimally non-totally unimodular matrices to occur only in small orders.\footnote{Private communication of the second author with K.~Truemper in 2024.}
Later, it was conjectured in~\cite[Conjecture~3.14]{Chervet_Grappe_Vallee_2026} that the only possible orders are $3$ and $5$.
We prove this conjecture.

\begin{theorem}\label{thm:main}
Up to row and column permutations and complement operations, the only complement minimally non-totally unimodular matrices are
$$
    C_3=\begin{bmatrix}
        1 & 1 & 0\\
        0& 1 & 1\\
        1&0&1
    \end{bmatrix} \wideand
    C_5=\begin{bmatrix}
        1&1&0&0&0\\
        0&1&1&0&0\\
        0&0&1&1&0\\
        0&0&0&1&1\\
        1&0&0&0&1
    \end{bmatrix}.
$$
\end{theorem}

The core correspondence gives the following corollary.
The relevant definitions are given in Section~\ref{sec:cores}.

\begin{corollary}\label{cor:te-interlace}
Every full-dimensional thick te-interlace is equivalent to exactly one of
$$
    H_4= \begin{bmatrix}
        1&1&1&1\\
        1&-1&-1&1\\
        1&1&-1&-1\\
        1&-1&1&-1
    \end{bmatrix} \wideand
    H_6= \begin{bmatrix}
        1&1&1&1&1&1\\
        1&-1&-1&1&1&1\\
        1&1&-1&-1&1&1\\
        1&1&1&-1&-1&1\\
        1&1&1&1&-1&-1\\
        1&-1&1&1&1&-1
    \end{bmatrix}.
$$
In particular, no full-dimensional thick te-interlace has size at least~$8$.
\end{corollary}

The classification also has a geometric consequence.
Combining Corollary~\ref{cor:te-interlace} with \cite[Corollary~3.4 and Theorem~4.5]{Chervet_Grappe_Vallee_2026} gives the following corollary.

\begin{corollary}\label{cor:triangulation}
    Every te-cone, that is, every simplicial cone generated by the rows of a totally equimodular matrix, admits a regular unimodular Hilbert triangulation.
\end{corollary}

The proof of Theorem~\ref{thm:main} uses matroid theory.
Let $B$ be a complement minimally non-totally unimodular matrix.
We construct a binary matroid $N_B$ represented over~$\F_2$ by a matrix of the form $[I\mid D(B)]$.
This construction is closely related to the matrices used in Truemper's almost-representation theorem~\cite[Construction~12.3.7]{Truemper1992book}.
The matroid $N_B$ has two complementary circuit-hyperplanes.
Relaxing one of them gives a matroid~$M_B$.

We show that $M_B$ is neither binary nor ternary in Section~\ref{sec:not-binary-ternary}, while every proper minor is binary or ternary in Section~\ref{sec:minors}.
Thus $M_B$ is an excluded minor for the class of matroids that are binary or ternary.
Using the excluded-minor classification in~\cite{Mayhew-Oporowski-Oxley-Whittle2011}, we show that $m\in \{3,5\}$ and identify the corresponding excluded minors.
Finally, we classify in Section~\ref{sec:proof} the two remaining cases using inverses and determinants.

\section{Thick te-interlaces and their cores} \label{sec:cores}

We use the terminology in~\cite{Chervet_Grappe_Vallee_2026}.
Two matrices are called \emph{equivalent} if one can be obtained from the other by permuting rows and columns and by multiplying rows and columns by~$-1$.
A \emph{full-dimensional thick te-interlace} of size~$n$ is an invertible matrix $H\in\{\pm1\}^{n\times n}$ such that
$$
    |\det H|=2^n
$$
and every nonzero proper $k\times k$ minor has absolute value $2^{k-1}$.

Let $H$ be a full-dimensional thick te-interlace.
After resigning rows and columns, we may suppose that the first row and the first column are all ones.
Then
\begin{equation}\label{eq:block-H}
    H=H(B) \coloneq 
        \begin{bmatrix}
            1 & \one^{\sfT} \\ 
            \one & \one\one^{\sfT}-2B
        \end{bmatrix},
\end{equation}
where $B\in\{0,1\}^{m\times m}$ and $m=n-1$.
The matrix~$B$ is called the \emph{core} of~$H$.

A real matrix is \emph{totally unimodular} if every square subdeterminant is $0$, $1$, or $-1$.
A $0,1$ matrix is \emph{minimally non-totally unimodular} if it is not totally unimodular but every proper submatrix is totally unimodular.
For a $0,1$ matrix~$B$, a \emph{row-$i$ complement} keeps row~$i$ and adds row~$i$ modulo two to every other row.
Column complements are defined in the same fashion.
The set of matrices obtained by sequences of row and column complements is the \emph{complement orbit} $\cO(B)$.
The matrix~$B$ is \emph{complement minimally non-totally unimodular} if every matrix in~$\cO(B)$ is minimally non-totally unimodular.

We use the following core correspondence.

\begin{proposition}[{\cite[Corollary~3.35 and Theorem~3.26]{Chervet_Grappe_Vallee_2026}}]\label{prop:core}
    The matrix $H(B)$ in~\eqref{eq:block-H} is a full-dimensional thick te-interlace if and only if $B$ is complement minimally non-totally unimodular.
    Moreover, the order~$m$ of~$B$ is odd.
\end{proposition}
This proposition explains why Theorem~\ref{thm:main} and Corollary~\ref{cor:te-interlace} are equivalent.

We also need the following simple relation between equivalent matrices and the complement orbit; the explicit formulas are computed in~\cite[Lemma~3.32]{Chervet_Grappe_Vallee_2026}.

\begin{lemma}\label{lem:core-equivalence}
    If $B'$ is obtained from $B$ by row or column complements and by row or column permutations, then $H(B')$ is equivalent to $H(B)$.
\end{lemma}
\begin{proof}
    Suppose first that $B'$ is the row-$i$ complement of~$B$.
    In $H(B)$, swap the first row with row~$i+1$.
    Then resign the columns so that the new first row is all ones.
    The new core is the row-$i$ complement of~$B$, up to a row permutation.
    The column case follows by transposing this argument.
    Row and column permutations of the core clearly come from row and column permutations of $H(B)$.
\end{proof}

Since $B$ itself is minimally non-totally unimodular, Camion's theorem~\cite{Camion1965} gives
\begin{equation}\label{eq:camion}
    \det B=\pm2, \qquad B^{-1}\in\left\{\pm\frac12\right\}^{m\times m},
\end{equation}
and every row and every column of~$B$ has a positive even number of ones.
In particular, every cofactor of~$B$ is~$\pm1$.
Also, $m\neq1$, so $m\geq3$.

Over~$\F_2$ we obtain the following facts.
\begin{lemma}\label{lem:mod2-B}
    Over~$\F_2$, we have
    $$
        B\one=0, \quad \one^{\sfT}B=0, \quad \rank B=m-1, \wideand \ker B=\langle\one\rangle.
    $$
\end{lemma}
\begin{proof}
    The first two equalities follow from the even row and column sums.
    Since~$\det B$ is even, $B$ is singular over~$\F_2$.
    Every cofactor is odd, so $B$ has rank~$m-1$ over~$\F_2$.
    Since~$B\one=0$, its kernel is spanned by~$\one$.
\end{proof}

\section{The relaxed matroid} \label{sec:not-binary-ternary}

We use the standard notations and definitions from matroid theory; see~\cite{Oxley2011book}.
Put $n=m+1$.
Let $N_B$ be the binary matroid represented by the matrix $[I_n\mid D(B)]$, where $D(B)=\begin{bmatrix} 0 & \one^{\sfT} \\ \one & B\end{bmatrix}$.
Label the identity columns by $r_0,r_1,\ldots,r_m$ and the columns of~$D(B)$ by $c_0,c_1,\ldots,c_m$.
Thus
$$
    c_0=\begin{bmatrix}0\\ \one\end{bmatrix} \wideand c_j=\begin{bmatrix}1\\ B^j\end{bmatrix} \quad(1\leq j\leq m),
$$
where~$B^j$ is column~$j$ of~$B$.

Define $\cJ=\{c_0,r_1,\ldots,r_m\}$ and $\cK=\{r_0,c_1,\ldots,c_m\}$.
These two sets partition the ground set of~$N_B$.

\begin{lemma}\label{lem:two-CH}
    The sets~$\cJ$ and~$\cK$ are circuit-hyperplanes of~$N_B$.
\end{lemma}
\begin{proof}
    We have
    $$
    c_0=r_1+\cdots+r_m.
    $$
    This is a minimal relation, so~$\cJ$ is a circuit.
    Its span is the hyperplane~$x_0=0$, so it is also a hyperplane.
    
    Next, since~$m$ is odd and~$B\one=0$ over~$\F_2$,
    $$
        c_1+\cdots+c_m = \begin{bmatrix} m\\ B\one \end{bmatrix} = \begin{bmatrix} 1\\ 0 \end{bmatrix} =r_0.
    $$
    The columns~$c_1,\ldots,c_m$ are independent.
    Indeed, if~$\sum_j\lambda_jc_j=0$, then~$B\lambda=0$.
    By Lemma~\ref{lem:mod2-B}, either~$\lambda=0$ or~$\lambda=\one$.
    The second case is impossible because the first coordinate of~$\sum_jc_j$ is one.
    Thus~$\cK$ is a circuit.
    
    Finally, let $\ell(x_0,x_1,\ldots,x_m)=x_1+\cdots+x_m$.
    The even column sums of~$B$ give~$\ell(c_j)=0$ for~$j\geq1$, and also~$\ell(r_0)=0$.
    Thus~$\cK$ lies in the hyperplane~$\ker\ell$.
    Since~$\cK$ has rank~$n-1$, it is a hyperplane.
\end{proof}

Let~$M_B$ be obtained from~$N_B$ by relaxing the circuit-hyperplane~$\cJ$.
Thus the bases of~$M_B$ are the bases of~$N_B$, together with the one new basis~$\cJ$.

\begin{lemma}\label{lem:not-binary}
    The matroid~$M_B$ is not binary.
\end{lemma}
\begin{proof}
    The circuit~$\cK$ remains a circuit in~$M_B$, since the relaxation only adds~$\cJ$ as a basis.
    For each~$x\in\cK$, the set~$\cJ\cup\{x\}$ is a circuit of~$M_B$: after the relaxation,~$\cJ$ is a basis, and for every~$y\in\cJ$ the set~$(\cJ-\{y\})\cup\{x\}$ is already a basis of~$N_B$.
    
    Take distinct~$x,y\in\cK$.
    If~$M_B$ were binary, the symmetric difference of the two circuits $\cJ\cup\{x\}$ and $\cJ\cup\{y\}$ would be a disjoint union of circuits.
    Their symmetric difference is~$\{x,y\}$.
    But~$\{x,y\}$ is independent, since it is a proper subset of the circuit~$\cK$.
    This is a contradiction.
\end{proof}

We now state the part of Truemper's almost-representation theorem that we need.
For a~$p\times q$~$0,1$ matrix~$U$, put
\begin{equation}\label{eq:D2D3}
    D_2(U)=
        \begin{bmatrix}
            0&\one_q^{\sfT}\\
            \one_p&U
        \end{bmatrix} \quad\text{over }\F_2, \wideand 
    D_3(U)=
        \begin{bmatrix}
            -1&\one_q^{\sfT}\\
            \one_p&U
        \end{bmatrix} \quad\text{over }\F_3.
\end{equation}

\begin{proposition}[{\cite[Theorem~12.3.8]{Truemper1992book}}]\label{prop:truemper}
    The following statements hold.
    \begin{enumerate}
    \item If~$U$ is complement totally unimodular, then the matroids represented by~$[I_{p+1}\mid D_2(U)]$ over~$\F_2$ and by~$[I_{p+1}\mid D_3(U)]$ over~$\F_3$ have the same bases except for the basis corresponding to the upper-left~$1\times1$ entry of $D_\alpha(U)$, $\alpha=2,3$.
    \item Conversely, suppose a ternary matrix and~$D_2(U)$ have the same nonsingular square submatrices except for the upper-left~$1\times1$ submatrix.
    If there are no zero rows or columns, then~$U$ is complement totally unimodular.
    \end{enumerate}
\end{proposition}

\begin{lemma}\label{lem:not-ternary}
    The matroid~$M_B$ is not ternary.
\end{lemma}
\begin{proof}
    Assume that~$M_B$ is ternary.
    The set~$\{r_0,\ldots,r_m\}$ is a basis of both~$N_B$ and~$M_B$, so choose a ternary reduced representation~$[I_n\mid R]$ with this basis.
    For a reduced representation, a square submatrix is nonsingular exactly when the corresponding exchange of identity columns gives a basis.
    The bases of~$N_B$ and~$M_B$ differ only in~$\cJ$.
    Therefore the nonsingularity patterns of~$R$ over~$\F_3$ and~$D(B)$ over~$\F_2$ differ only at the upper-left~$1\times1$ submatrix.
    
    The matroid~$M_B$ has no loops or coloops.
    This follows directly from the representation, from the positive row and column sums of~$B$, and from the two bases~$\{r_0,\ldots,r_m\}$ and~$\cJ$.
    Hence~$R$ has no zero row or column.
    The matrix~$D(B)$ also has no zero row or column.
    Proposition~\ref{prop:truemper} now says that~$B$ is complement totally unimodular.
    This is impossible because~$B$ is minimally non-totally unimodular.
    Hence~$M_B$ is not ternary.
\end{proof}

\section{The excluded minor} \label{sec:minors}

The following is a direct consequence of~\cite[Corollary~3.35 and~3.33]{Chervet_Grappe_Vallee_2026}, but we give a simple proof for completeness.

\begin{lemma}\label{lem:proper-CTU}
    Let~$B$ be complement minimally non-totally unimodular.
    If~$B'$ belongs to~$\cO(B)$ and~$U$ is a proper submatrix of~$B'$, then~$U$ is complement totally unimodular.
\end{lemma}

\begin{proof}
    Every row or column complement of~$U$ can be obtained by applying the same complement to~$B'$ and then restricting to the rows and columns of~$U$.
    The resulting larger matrix still belongs to~$\cO(B)$, so it is minimally non-totally unimodular.
    Every proper submatrix of it is totally unimodular.
    Hence every matrix in~$\cO(U)$ is totally unimodular.
\end{proof}

We use the standard behavior of relaxation under deletion and contraction.
If~$e\in\cJ$, then
$$
    M_B\setminus e=N_B\setminus e,
$$
while~$M_B/e$ is obtained from~$N_B/e$ by relaxing~$\cJ-\{e\}$.
If~$e\in\cK$, then
$$
    M_B/e=N_B/e,
$$
while~$M_B\setminus e$ is obtained from~$N_B\setminus e$ by relaxing~$\cJ$.
Thus one of the two one-element minors is always binary.
We show that the other one is ternary.

A direct pivot computation in the reduced binary representation gives the following table.
Here~$B_{[i]}$ is the row-$i$ complement and~$B^{[j]}$ is the column-$j$ complement.

\begin{center}
    \renewcommand{\arraystretch}{1.25}
    \begin{tabular}{c c c}
        \toprule
        Element~$e$ & relaxed one-element minor & core~$U_e$\\
        \midrule
        $r_i$, $1\leq i\leq m$ & $M_B/r_i$ & $B$ with row~$i$ deleted\\
        $c_0$ & $M_B/c_0$ & $B_{[i]}$ with row~$i$ deleted\\
        $c_j$, $1\leq j\leq m$ & $M_B\setminus c_j$ & $B$ with column~$j$ deleted\\
        $r_0$ & $M_B\setminus r_0$ & $B^{[j]}$ with column~$j$ deleted\\
        \bottomrule
    \end{tabular}
\end{center}

In the second row, choose any~$i$ and pivot on the entry in row~$r_i$ and column~$c_0$ before contracting~$c_0$.
In the fourth row, choose any~$j$ and pivot on the entry in row~$r_0$ and column~$c_j$ before deleting~$r_0$.
After a row and column permutation, the resulting binary reduced matrix is~$D_2(U_e)$, and its relaxed circuit-hyperplane is the exceptional basis in Proposition~\ref{prop:truemper}.
The matrices~$U_e$ may be rectangular; this is allowed in~\eqref{eq:D2D3}.

Each~$U_e$ is a proper submatrix of a matrix in~$\cO(B)$.
By Lemma~\ref{lem:proper-CTU}, it is complement totally unimodular.
Therefore~$D_3(U_e)$ gives a ternary representation of the relaxed minor.

\begin{proposition}\label{prop:excluded}
    Every proper minor of~$M_B$ is binary or ternary, but~$M_B$ itself is neither binary nor ternary.
    Hence~$M_B$ is an excluded minor for the class of matroids that are binary or ternary.
\end{proposition}

\begin{proof}
    Every one-element deletion and contraction is binary or ternary by the argument above.
    The class of matroids that are binary or ternary is minor-closed.
    Since every proper minor is a minor of a one-element minor, every proper minor has the same property.
    The last claim follows from Lemmas~\ref{lem:not-binary} and~\ref{lem:not-ternary}.
\end{proof}

Mayhew, Oporowski, Oxley, and Whittle gave the following list of excluded minors.

\begin{theorem}[{\cite[Theorem~1.1]{Mayhew-Oporowski-Oxley-Whittle2011}}] \label{thm:MOOW}
    The excluded minors for the class of matroids that are binary or ternary are
    \begin{equation}\label{eq:excluded-list}
        \begin{split}
            &U_{2,5},\quad U_{3,5},\quad U_{2,4}\oplus F_7,\quad U_{2,4}\oplus F_7^*,\\
            &U_{2,4}\oplus_2 F_7,\quad U_{2,4}\oplus_2 F_7^*,\quad \AG(3,2)',\quad T_{12}'.
        \end{split}
    \end{equation}
\end{theorem}

\begin{proposition}\label{prop:sizes}
    A complement minimally non-totally unimodular matrix has size~$3$ or $5$.
\end{proposition}
\begin{proof}
    Suppose $B$ is complement minimally non-totally unimodular of size $m$.
    The matroid~$M_B$ has rank~$n=m+1$ and
    $$
        |E(M_B)|=2n.
    $$
    By Proposition~\ref{prop:excluded}, it is one of the matroids in~\eqref{eq:excluded-list}.
    The first six matroids in the list have respectively~$5,\ 5,\ 11,\ 11,\ 9$, and~$9$ elements.
    These numbers are odd.
    The last two matroids have~$8$ and~$12$ elements.
    Since~$|E(M_B)|=2n$ is even, we must have $2n\in\{8,12\}$.
    Therefore~$n\in\{4,6\}$, that is $m\in\{3,5\}$.
\end{proof}

The preceding argument also identifies the excluded minor~$M_B$ precisely.

\begin{proposition}\label{prop:identify}
    Let~$B$ be a complement minimally non-totally unimodular matrix of order~$m$.
    If~$m=3$, then
    $$
        M_B\cong \AG(3,2)'.
    $$
    If~$m=5$, then
    $$
        M_B\cong T_{12}'.
    $$
\end{proposition}
\begin{proof}
    By Proposition~\ref{prop:excluded}, $M_B$ is one of the matroids in~\eqref{eq:excluded-list}.
    Since $|E(M_B)|=2(m+1)$, the cases $m=3$ and $m=5$ give respectively $8$ and $12$ elements.
    The list~\eqref{eq:excluded-list} contains exactly one matroid on $8$ elements, namely $\AG(3,2)'$, and exactly one matroid on $12$ elements, namely $T_{12}'$.
    The result follows.
\end{proof}

\section{The classification} \label{sec:proof}

It remains to identify the two small cases.
We first obtain one useful signed equation.
Let
$$
    Q=2B^{-1}\in\{\pm1\}^{m\times m}, \quad u=Q\one, \wideand \sigma=\one^{\sfT}Q\one.
$$
A block inverse calculation gives
$$
    4H(B)^{-1}= \begin{bmatrix}
        4-\sigma&\one^{\sfT}Q\\
        Q\one&-Q
    \end{bmatrix}.
$$
By~\cite[Corollary 3.38]{Chervet_Grappe_Vallee_2026}, $4H(B)^{-1}$ is again a thick te-interlace.
In particular, all its entries are~$\pm1$.
Hence
\begin{equation}\label{eq:u-sigma}
    u\in\{\pm1\}^m \wideand \sigma\in\{3,5\}.
\end{equation}
Since~$BQ=2I$, we also have
\begin{equation}\label{eq:signed-degree}
    Bu=2\one.
\end{equation}

For~$r\geq2$, let~$P_r$ be the permutation matrix of an~$r$-cycle and put $C_r=I_r+P_r$.
The support graph of~$C_r$ is a cycle with~$r$ row vertices and~$r$ column vertices.
Since~$\det(I_r+P_r)=1-(-1)^r$, we have
\begin{equation}\label{eq:cycle-det}
    |\det C_r|= \begin{cases}
        2,&r\text{ odd},\\
        0,&r\text{ even}.
    \end{cases}
\end{equation}

We are now ready to prove Theorem~\ref{thm:main}.

\begin{proof}[Proof of Theorem~\ref{thm:main}]
    By Proposition~\ref{prop:sizes}, we know that $m\in\{3,5\}$.

    First suppose~$m=3$.
    Since~$\sigma$ is the sum of the three entries of~$u$, \eqref{eq:u-sigma} gives~$\sigma=3$ and hence~$u=\one$.
    Equation~\eqref{eq:signed-degree} says that every row of~$B$ has two ones.
    Every column has a positive even number of ones, so every column also has two ones.
    In particular, up to permutations of rows and columns, we have $B=C_3$.
    
    Now suppose~$m=5$.
    If~$\sigma=5$, then again~$u=\one$, and every row of~$B$ has two ones.
    If~$\sigma=3$, then~$u$ has exactly one negative entry.
    After a column permutation, suppose~$u_j=-1$ and~$u_k=1$ for~$k\neq j$.
    If row~$i$ of~$B$ has $w_i$ ones, then~\eqref{eq:signed-degree} gives
    $$
        w_i-2B_{ij}=2.
    $$
    Thus~$w_i=2$ when~$B_{ij}=0$, and~$w_i=4$ when~$B_{ij}=1$.
    Apply the column-$j$ complement.
    Rows with~$B_{ij}=0$ do not change.
    In a row with~$B_{ij}=1$, the other four entries are toggled, so it has now two ones.
    We have therefore found a matrix~$B'\in\cO(B)$ in which every row has two ones.
    
    In both cases, we now have a matrix~$B'\in\cO(B)$ of order~$m$ with every row containing two ones.
    Since~$B'$ is minimally non-totally unimodular, every column has a positive even number of ones.
    The total number of ones is~$2m$, so every column also has two ones.
    The bipartite support graph of~$B'$ is therefore~$2$-regular.
    It is a disjoint union of cycles.
    After row and column permutations,~$B'$ is block diagonal with blocks~$C_{r_1},\ldots,C_{r_t}$.
    By Camion's theorem, $|\det B'|=2$.
    Formula~\eqref{eq:cycle-det} shows that there is exactly one block.
    Hence~$B'$ is row-column permutation equivalent to~$C_m$.
    
    Conversely, a direct check shows that $H(C_3)$ and $H(C_5)$ are thick te-interlaces, so Proposition~\ref{prop:core} gives the converse.
\end{proof}

\begin{proof}[Proof of Corollary~\ref{cor:te-interlace}]
    Let $H$ be a thick te-interlace of size $n$.
    By Proposition~\ref{prop:core}, the core $B$ of $H$ is a complement minimally non-totally unimodular matrix.
    By Theorem~\ref{thm:main}, we know that up to complement operations and permutation of rows and columns, $B$ is equivalent to $C_3$ or $C_5$, which are the cores of $H_4$ or $H_6$, respectively.
    Therefore, by Lemma~\ref{lem:core-equivalence} every thick te-interlace is equivalent to either $H_4$ or $H_6$.
\end{proof}

\begin{proof}[Proof of Corollary~\ref{cor:triangulation}]
    Theorem~4.5 of~\cite{Chervet_Grappe_Vallee_2026} proves this whenever there is no thick te-interlace of size greater than six.
    Any thick te-interlace contains an invertible square submatrix with the same number of rows, and this submatrix is a full-dimensional thick te-interlace; see~\cite[Corollary~3.34]{Chervet_Grappe_Vallee_2026}.
    Theorem~\ref{thm:main} therefore shows that a thick te-interlace has size at most six.
\end{proof}

\subsection*{Acknowledgement}
The second author is grateful to Klaus Truemper for fruitful discussions on complement matrices and for encouraging him to work on this conjecture.
He also thanks an anonymous referee of~\cite{Chervet_Grappe_Vallee_2026} for suggesting a matroid-theoretic approach to this problem and Stefan K\"ober for helpful discussions on minimally non-totally unimodular matrices.

\end{document}